\documentclass[11pt, reqno]{amsart}
\usepackage[utf8]{inputenc}
\usepackage{indentfirst, amssymb, amsmath, amsthm, mathrsfs, setspace, indentfirst, enumerate,  mathrsfs, amsmath, amsthm}
\usepackage[colorlinks=true,linkcolor=purple, citecolor=blue,urlcolor=magenta]{hyperref}
\usepackage{graphicx}      
\usepackage{float}         
\usepackage{xcolor}        
\usepackage{colortbl}     
\usepackage{multirow}     
\usepackage{tikz}          
\definecolor{sectionlink}{RGB}{0,100,200} 
\newtheorem{theo}{Theorem}[section]
\newtheorem{lem}{Lemma}[section]

\newtheorem{exm}{Example}[section]
\newtheorem{defi}{Definition}[section]

\newtheorem*{theoA}{Theorem A}

\numberwithin{equation}{section}
\newcommand{\beas}{\begin{eqnarray*}}
\newcommand{\eeas}{\end{eqnarray*}}
\newcommand{\bea}{\begin{eqnarray}}
\newcommand{\eea}{\end{eqnarray}}

\begin{document}

\title[On the Sharp Bound of the Hankel Determinant for $\beta-$ Spirallike...]{On Hankel Determinants of $\beta$-Spirallike Mappings in Complex Banach Spaces}
\author[ N. Sarkar and P. Das]{Nabadwip Sarkar and Pradip Das}
\address{Amity School of Applied Sciences, Amity University Mumbai, Panvel, Navi Mumbai, Maharashtra-410206, India}
\email{nsarkar@mum.amity.edu, nabadwipsarkar52@gmail.com}
\address{Department of Mathematics, Raiganj University, Raiganj, West Bengal-733134, India.}
\email{pradipsmath@gmail.com}

\makeatletter
\@namedef{subjclassname@2020}{\textup{2020} Mathematics Subject Classification}
\makeatother


\begin{abstract} This paper is devoted to the study of the second Hankel determinant for normalized $\beta$-spirallike mappings on the unit ball of a complex Banach space. Using the relationship between Banach-space holomorphic mappings and corresponding one-variable $\beta$-spirallike functions, we first establish sharp estimates for the initial homogeneous expansion coefficients. These coefficient inequalities are subsequently applied to derive a sharp bound for the second Hankel determinant. 
\end{abstract}

\maketitle

\section{Introduction and Preliminaries}

Let $\mathcal{A}$ denote the family of analytic functions in the open unit disk $\mathbb{U}=\{z\in\mathbb{C}:|z|<1\},$
normalized by
\begin{equation}\label{eq:A}
f(z)=z+\sum_{m=2}^{\infty}a_mz^m.
\end{equation}
Furthermore, let $\mathcal{S}\subset\mathcal{A}$ be the class of normalized univalent functions, and let $\mathcal{K}$ denote the subclass of $\mathcal{S}$ consisting of convex functions.

Throughout this paper, we shall frequently use the classical Carath\'eodory class, denoted by $\mathcal{P}$, which consists of analytic functions $p$ satisfying
\[
p(0)=1
\quad\text{and}\quad
\operatorname{Re}p(z)>0,\qquad z\in\mathbb{U}.
\]
Every function $p\in\mathcal{P}$ admits the Taylor expansion
\begin{equation}\label{eq:P}
p(z)=1+\sum_{m=1}^{\infty}p_mz^m
=1+p_1z+p_2z^2+p_3z^3+\cdots,
\qquad z\in\mathbb{U}.
\end{equation}

Let $X$ be a complex Banach space endowed with the norm $\|\cdot\|$, and let
\[
\mathbb{B}=\{x\in X:\|x\|<1\}
\]
be its open unit ball. We denote by $\mathcal{L}(X,Y)$ the Banach space of all bounded linear operators from $X$ into another complex Banach space $Y$. The identity operator on $X$ is denoted by $I$.

For each nonzero vector $x\in X$, define
\[
T(x)=
\left\{
T_x\in\mathcal{L}(X,\mathbb{C}):
T_x(x)=\|x\|,
\ \|T_x\|=1
\right\}.
\]
The Hahn--Banach theorem guarantees that $T(x)$ is nonempty. Moreover, for every $\xi\in\mathbb{C}\setminus\{0\}$,
\[
T_{\xi x}(\cdot)=\frac{|\xi|}{\xi}\,T_x(\cdot),
\]
which establishes a one-to-one correspondence between $T(\xi x)$ and $T(x)$.

Let $\mathcal{H}(\mathbb{B})$ denote the family of all holomorphic mappings from $\mathbb{B}$ into $X$. If $F\in\mathcal{H}(\mathbb{B})$, then for every $x\in \mathbb{B}$, the Fr\'echet--Taylor expansion is given by
\[
F(y)
=
\sum_{n=0}^{\infty}
\frac{1}{n!}
D^nF(x)\bigl((y-x)^n\bigr),
\]
for all $y$ sufficiently close to $x$, where
\[
D^nF(x)\bigl((y-x)^n\bigr)
=
D^nF(x)(y-x,\ldots,y-x),
\]
and $D^nF(x)$ denotes the $n$th Fr\'echet derivative of $F$ at $x$. Each $D^nF(x)$ is a bounded symmetric $n$-linear operator from $X^n$ into $X$.

A mapping $F\in\mathcal{H}(\mathbb{B})$ is called \emph{biholomorphic} if $F(\mathbb{B})$ is a domain in $X$ and the inverse mapping
\[
F^{-1}:F(\mathbb{B})\rightarrow \mathbb{B}
\]
exists and is holomorphic. Moreover, $F$ is said to be \emph{locally biholomorphic} whenever the Fr\'echet derivative $DF(x)$ is invertible with bounded inverse for every $x\in \mathbb{B}$.

A holomorphic mapping $F:B\rightarrow X$ is said to be \emph{normalized} if
\[
F(0)=0
\quad\text{and}\quad
DF(0)=I.
\]
Several important subclasses of normalized biholomorphic mappings in Banach spaces can be found in the monograph~\cite{GK2003}.

For later use, let $x_0\in X$ satisfy $\|x_0\|=1$ and choose $T_{x_0}\in T(x_0)$. We introduce the quantities $A_1=1$
and, for $n=2,3,4,5$,
\begin{equation}\label{eq:An}
A_n=\frac{1}{n!}T_{x_0}\left(D^nF(0)(x_0,\ldots,x_0)
\right),
\end{equation}
where the vector $x_0$ appears exactly $n$ times in the multilinear form $D^nF(0)$.

We shall also require the following definitions.

\begin{defi}\cite{GK2003}
Let $|\beta|<\pi/2$. A normalized locally biholomorphic mapping $F:\mathbb{B}\rightarrow X$ is said to be \emph{$\beta$-spirallike} if
\begin{equation}\label{eq:spiral-banach}
\operatorname{Re}\left(
e^{-i\beta}
T_x\bigl((DF(x))^{-1}F(x)\bigr)
\right)>0,
\qquad
x\in B\setminus\{0\},
\quad
T_x\in T(x).
\end{equation}
The family of all $\beta$-spirallike mappings on $\mathbb{B}$ will be denoted by $\widehat{\mathcal{S}}_{\beta}(\mathbb{B}).$

\medskip
In particular, when $X=\mathbb{C}$ and $\mathbb{B}=\mathbb{U}$, we simply write $\widehat{\mathcal{S}}_{\beta}=\widehat{\mathcal{S}}_{\beta}(\mathbb{U}).$
In this case, condition \eqref{eq:spiral-banach} reduces to
\[
\operatorname{Re}\left(
e^{i\beta}
\frac{zf'(z)}{f(z)}
\right)>0,
\qquad z\in\mathbb{U},
\]
which coincides with the classical definition of $\beta$-spirallike functions in the unit disk.
\end{defi}

\begin{defi}\cite{Chirila2014, XL2009}
Assume that $|\beta|<\pi/2$, and let $g:\mathbb{U}\rightarrow\mathbb{C}$
be a biholomorphic function satisfying $g(0)=1, \operatorname{Re}g(z)>0, z\in\mathbb{U}.$
Define
\[
g_{\beta}(z)
=\frac{\cos\beta\,g(z)+i\sin\beta}{e^{i\beta}},\qquad z\in\mathbb{U}.
\]

The class $\mathcal{M}_{g}(\beta)$ consists of all holomorphic mappings
$H:B\rightarrow X$ satisfying
\begin{equation}\label{eq:Mg}
\frac{\|x\|}{T_x(H(x))}
\in
g_{\beta}(\mathbb{U}),
\qquad
x\in B\setminus\{0\},
\quad
T_x\in T(x).
\end{equation}

If $X=\mathbb{C}$ and $B=\mathbb{U}$, then \eqref{eq:Mg} is equivalent to
\[
\frac{z}{H(z)}
\in
g_{\beta}(\mathbb{U}),
\qquad z\in\mathbb{U}.
\]
When $\beta=0$, the class $\mathcal{M}_{g}(0)$ reduces to the class introduced in
\cite{GHK2002,Kohr1998} and subsequently investigated in \cite{GHKK2017}-\cite{HKK2021}.
Furthermore, choosing
\[
g(z)=\frac{1+z}{1-z},
\qquad z\in\mathbb{U},
\]
one observes that $(DF(x))^{-1}F(x)\in\mathcal{M}_{g}(\beta)$ if and only if $F\in\widehat{\mathcal{S}}_{\beta}(\mathbb{B}).$
\end{defi} 
The Hankel determinant \(H_{q,n}(f)\) of a function \(f\in\mathcal{A}\) given by \eqref{eq:A} is defined by
\[
H_{q,n}(f)
=
\begin{vmatrix}
a_n & a_{n+1} & \cdots & a_{n+q-1}\\
a_{n+1} & a_{n+2} & \cdots & a_{n+q}\\
\vdots & \vdots & \ddots & \vdots\\
a_{n+q-1} & a_{n+q} & \cdots & a_{n+2(q-1)}
\end{vmatrix},
\]
where \(a_1=1\) and \(n,q\in\mathbb{N}\).

In particular, the second-order Hankel determinants are given by
\[
H_{2,2}(f)
=
\begin{vmatrix}
a_2 & a_3\\
a_3 & a_4
\end{vmatrix}
=
a_2a_4-a_3^2,
\]

and

\[
H_{2,1}(f)
=
\begin{vmatrix}
a_1 & a_2\\
a_2 & a_3
\end{vmatrix}
=
a_3-a_2^2.
\]
In recent years, the problem of determining sharp upper bounds for the Hankel determinant
\(|H_{q,n}(f)|\) has attracted considerable attention in geometric function theory. In particular,
the functional $H_{2,1}(f)=a_3-a_2^2,$ coincides with the classical Fekete--Szeg\"o functional introduced by Fekete and Szeg\"o
\cite{FS1933}. For the class $\mathcal{S}$, this functional was first estimated by Bieberbach
(see \cite[Vol.~I, p.~35]{Goodman1983}). Subsequently, Pommerenke \cite{Pommerenke1966} established a fundamental
result for the class $\mathcal{S}$, which stimulated extensive research on analogous coefficient
problems for various subclasses of univalent functions. More recently, considerable effort has
been devoted to obtaining sharp estimates for the second Hankel determinant $H_{2,2}(f)=a_2a_4-a_3^2,$
and several significant results have been reported in the literature (see, for example,
\cite{CKKLS2018,CKKLS2017,XDL2026,XHX2025}).\\

\medskip
In 2015, Krishna and Reddy \cite{KR2015} established the following sharp estimate for the second Hankel determinant of spirallike functions.

\begin{theoA}\cite[Theorem 3.1]{KR2015}
If $f(z)=z+\sum_{n=2}^{\infty}a_nz^n
\in\widehat{\mathcal{S}}_{\beta},
\;|\beta|\le \frac{\pi}{3},$
then
\[
|a_2a_4-a_3^2|\le \cos^2\beta,
\]
and the estimate is sharp.
\end{theoA}

A natural question is whether the above sharp estimate can be extended from the classical one-dimensional setting to spirallike mappings defined on complex Banach spaces. Motivated by this problem, the main objective of the present paper is to establish an analogue of the above result in the framework of complex Banach spaces. In particular, we derive a sharp upper bound for the second Hankel determinant of a class of spirallike mappings in Banach spaces, thereby extending the theorem of Krishna and Reddy to the infinite-dimensional setting.
 
\section{{\bf Lemmas}}
Before proving the sharpness of the main theorem, we establish the following auxiliary lemmas. These results will be used repeatedly in the subsequent analysis and form the basis for the proof of the sharpness of the obtained estimate.

\begin{lem}\label{L1}
Let $u\in X$ with $\|u\|=1$ and let $T_u\in T(u)$. Define $F(x)=\frac{f(T_u(x))}{T_u(x)}x, x\in \mathbb{B},$ where $f\in\mathcal{S}$. If $f$ is a normalized $\beta$-spirallike function on $\mathbb{U}$, then $F$ is a $\beta$-spirallike mapping on $\mathbb{B}$.
\end{lem}

\begin{proof}
Since $f$ is a normalized $\beta$-spirallike function on $\mathbb{U}$, we have
\begin{equation}
\operatorname{Re}\left(e^{i\beta}\frac{\xi f'(\xi)}{f(\xi)}\right)>0,
\qquad \xi\in\mathbb{U}.
\label{eq:spiral}
\end{equation}

For convenience, let $h(x)=\frac{f(T_u(x))}{T_u(x)},\; x\in\mathbb{B}.$ It follows immediately from the definition of $F$ that $F(x)=h(x)x.$

By the product rule,
\[
DF(x)y=h(x)y+Dh(x)y\,x,\qquad y\in X.
\]
Moreover,
\[
h(x)+Dh(x)x=f'(T_u(x)).
\]
Indeed,
\[
Dh(x)x
=
\frac{T_u(x)f'(T_u(x))-f(T_u(x))}
{T_u(x)},
\]
which immediately yields the above identity.

Next,
\[
Dh(x)F(x)
=
h(x)Dh(x)x,
\]
and therefore,
\[
(DF(x))^{-1}F(x)
=
\frac1{h(x)}
\left(
F(x)-\frac{Dh(x)F(x)}
{h(x)+Dh(x)x}x
\right).
\]
Substituting the above relations, we obtain
\[
\begin{aligned}
(DF(x))^{-1}F(x)
=
x-\frac{Dh(x)x}
{h(x)+Dh(x)x}x=
\frac{h(x)}
{h(x)+Dh(x)x}x=
\frac{f(T_u(x))}
{T_u(x)f'(T_u(x))}x.
\end{aligned}
\]
Hence,
\begin{equation}
T_x\bigl((DF(x))^{-1}F(x)\bigr)
=
\frac{f(T_u(x))}
{T_u(x)f'(T_u(x))}
\|x\|.
\label{eq:Tx}
\end{equation}

Since $\operatorname{Re}\left(e^{i\beta}
\frac{\xi f'(\xi)}{f(\xi)}\right)>0,$
it follows that $
\operatorname{Re}\left(
e^{-i\beta}
\frac{f(\xi)}
{\xi f'(\xi)}
\right)>0,
\; \xi\in\mathbb{U},$
because $\operatorname{Re}\left(\frac1z\right)
=
\frac{\operatorname{Re}(z)}{|z|^2}>0.$
whenever $\operatorname{Re}(z)>0$.

Combining \eqref{eq:Tx} with \eqref{eq:spiral}, we deduce that
\[
\operatorname{Re}\left(
e^{-i\beta}
T_x\bigl((DF(x))^{-1}F(x)\bigr)
\right)
=
\operatorname{Re}\left(
e^{-i\beta}
\frac{f(T_u(x))}
{T_u(x)f'(T_u(x))}
\right)\|x\|
>0,
\qquad x\in \mathbb{B}\setminus\{0\}.
\]
Hence, by the definition of $\beta$-spirallike mappings on $\mathbb{B}$, $F$ is a $\beta$-spirallike mapping on $\mathbb{B}$.
\end{proof}

\begin{lem}\label{L2}
Suppose that $g\in H(\mathbb{B},\mathbb{C})$ with $g(0)=1$, and define $F(x)=g(x)x, x\in\mathbb{B}.$ Fix $x_{0}\in\partial\mathbb{B}$ and let
$f(\xi)=g(\xi x_{0})\,\xi,\xi\in\mathbb{U}.$
Then, for $|\beta|<\pi/2$,
\[
f\in\widehat{\mathcal{S}}_{\beta}
\quad\Longleftrightarrow\quad
F\in\widehat{\mathcal{S}}_{\beta}(\mathbb{B}).
\]
\end{lem}

\begin{proof}
Assume first that $F\in\widehat{\mathcal{S}}_{\beta}(\mathbb{B}).$ Then $F$ is locally biholomorphic on $\mathbb{B}$. Consequently,
\[
g(x)\neq0,
\qquad
g(x)+Dg(x)x\neq0,
\qquad x\in\mathbb{B}.
\]

Since
\[
DF(x)y=g(x)y+Dg(x)y\,x,
\]
a direct computation gives
\begin{equation}
[DF(x)]^{-1}
=
\frac1{g(x)}
\left(
I-
\frac{xDg(x)}
{g(x)+Dg(x)x}
\right).
\label{eq:inverse}
\end{equation}

Hence,
\[
(DF(x))^{-1}F(x)
=
\frac{g(x)}
{g(x)+Dg(x)x}\,x,
\qquad x\in\mathbb{B}.
\]
Applying the supporting functional $T_x$ yields
\begin{equation}
T_x\!\left((DF(x))^{-1}F(x)\right)
=
\frac{\|x\|\,g(x)}
{g(x)+Dg(x)x}.
\label{eq:Tx}
\end{equation}

On the other hand,
\[
f(\xi)=g(\xi x_0)\xi,
\]
and therefore
\begin{equation}
f'(\xi)
=
g(\xi x_0)+Dg(\xi x_0)\,\xi x_0.
\label{eq:fprime}
\end{equation}
Combining \eqref{eq:Tx} and \eqref{eq:fprime}, we obtain
\[
e^{-i\beta}
\frac{f(\xi)}
{\xi f'(\xi)}
=
e^{-i\beta}
T_{\xi x_0}
\left(
(DF(\xi x_0))^{-1}
F(\xi x_0)
\right).
\]
Since $F\in\widehat{\mathcal{S}}_{\beta}(\mathbb{B}),$
it follows that
\[
\operatorname{Re}
\left(
e^{-i\beta}
\frac{f(\xi)}
{\xi f'(\xi)}
\right)>0\Leftrightarrow 
\operatorname{Re}
\left(
e^{i\beta}
\frac{\xi f'(\xi)}
{f(\xi)}
\right)>0,
\]
and hence $f\in\widehat{\mathcal{S}}_{\beta}.$\\

\medskip

Conversely, suppose that $f\in\widehat{\mathcal{S}}_{\beta}.$
Then $\operatorname{Re}
\left(
e^{i\beta}
\frac{\xi f'(\xi)}
{f(\xi)}
\right)>0,$
which implies that
\[
\frac{\xi f'(\xi)}{f(\xi)}\neq0,
\qquad \xi\in\mathbb{U}.
\]
Consequently,
\[
g(x)+Dg(x)x\neq0,
\qquad x\in\mathbb{B}.
\]
By \eqref{eq:inverse}, $DF(x)$ is invertible for every
$x\in\mathbb{B}$, and therefore $F$ is locally biholomorphic.

Finally, using \eqref{eq:Tx} together with
\[
\operatorname{Re}
\left(
e^{-i\beta}
\frac{f(\xi)}
{\xi f'(\xi)}
\right)>0,
\]
we obtain
\[
\operatorname{Re}
\left(
e^{-i\beta}
T_x
\left(
(DF(x))^{-1}F(x)
\right)
\right)>0,
\qquad x\in\mathbb{B}\setminus\{0\}.
\]
Hence, $F\in\widehat{\mathcal{S}}_{\beta}(\mathbb{B}).$

This completes the proof.
\end{proof}

\section{{\bf Main Results and Their Proofs}}

In this section, we investigate the second Hankel determinant for the class $\widehat{\mathcal{S}}_{\beta}(\mathbb{B})$ of spirallike mappings on complex Banach spaces. By employing the coefficient relations established in the preceding section together with sharp estimates for Carath\'eodory functions, we obtain a sharp upper bound for the second Hankel determinant. Finally, an extremal mapping is constructed to verify the sharpness of the obtained result.
\begin{theo}\label{T1}
Let $g\in H(\mathbb{B},\mathbb{C})$ satisfy $g(0)=1$, and define $F(x)=g(x)x, x\in\mathbb{B}.$
Suppose that $F\in\widehat{\mathcal{S}}_{\beta}(\mathbb{B})$, $|\beta|\leq \frac{\pi}{3}$. Then, for every
$x_{0}\in X$ with $\|x_{0}\|=1$,
\[
|H_{2,2}(F)|\leq \cos^2 \beta,
\]
where $H_{2,2}(F)=
\begin{vmatrix}
A_{2} & A_{3} \\
A_{3} & A_{4} \\
\end{vmatrix},$
or equivalently, $H_{2,2}(F)
=A_2A_4-A_3^2,$
with  $A_{2},A_{3},A_{4}$ defined by
\eqref{eq:An}. Moreover, the estimate is sharp.
\end{theo}

\begin{proof}
Fix $x_{0}\in\partial\mathbb{B}$ and define $f(\xi)=g(\xi x_{0})\,\xi, \xi\in\mathbb{U}.$
Since $F\in\widehat{\mathcal{S}}_{\beta}(\mathbb{B})$, Lemma~\ref{L2}
implies that $f\in\widehat{\mathcal{S}}_{\beta}.$

On the other hand,
\[
f(\xi)
=
T_{x_{0}}\bigl(F(\xi x_{0})\bigr),
\]
and hence
\[
a_{2}
=
\frac{f''(0)}{2!}
=
\frac{1}{2!}
T_{x_{0}}
\left(
D^{2}F(0)(x_{0}^{2})
\right)
=
A_{2},
\]
\[
a_{3}
=
\frac{f^{(3)}(0)}{3!}
=
\frac{1}{3!}
T_{x_{0}}
\left(
D^{3}F(0)(x_{0}^{3})
\right)
=
A_{3},
\]
and \[
a_{4}
=
\frac{f^{(4)}(0)}{4!}
=
\frac{1}{4!}
T_{x_{0}}
\left(
D^{4}F(0)(x_{0}^{4})
\right)
=
A_{4}.
\]

Therefore,
\[
H_{2,2}(F)
=
H_{2,2}(f).
\]
Applying Theorem~A (see \cite[Theorem 3.1.]{KR2015}),
we obtain
\[
|H_{2,2}(F)|
=
|H_{3,1}(f)|
\leq
\cos^2 \beta.
\]
This completes the proof.
\end{proof}

The following example demonstrates that the estimate obtained in Theorem \ref{T1} is sharp.

\begin{exm}
Consider the mapping
\[
F(x)=
\frac{f\!\left(T_{x_0}(x)\right)}{T_{x_0}(x)}\,x,
\qquad
x_0\in\partial\mathbb{B},\;
T_{x_0}\in T(x_0),
\]
where
\[
f(z)=\frac{z}{\left(1+z^2\right)^{\cos\beta\,e^{-i\beta}}},
\qquad z\in\mathbb{U}.
\]

By Lemma \ref{L1}, we have $F\in\widehat{\mathcal{S}}_{\beta}(\mathbb{B}).$
Furthermore, a straightforward computation gives
\[
\frac{D^2F(0)(x^2)}{2!}=0,\;\;\;
\frac{D^3F(0)(x^3)}{3!}
=
-\cos\beta\,e^{-i\beta}
\bigl(T_{x_0}(x)\bigr)^2x,
\;\;\;\text{and}\;\;\;
\frac{D^4F(0)(x^4)}{4!}=0.
\]

Hence,
\[
A_2=0,\;\;
A_3=-\cos\beta\,e^{-i\beta},\;\;\text{and}\;\;
A_4=0.
\]
Therefore,
\[
|H_{2,2}(F)|
=
\left|A_2A_4-A_3^2\right|
=
\cos^2\beta.
\]

Thus, the upper bound obtained in Theorem \ref{T1} is attained, showing that the estimate is sharp.
\end{exm}

\medskip
Next, removing the restrictive assumption $F(x)=g(x)x$, we generalize Theorem~A to higher dimensions under weaker assumptions than those of Theorem \ref{T1}. Assume that
\begin{equation}\label{cc1}
\frac{D^{k+1}F(0)\left(x^{k+1}\right)}{(k+1)!}
= H_{F,k}(x)x,\qquad x\in X,\quad k=1,2,3.
\end{equation}
where $H_{F,k}(x)$ is a homogeneous polynomial of degree $k$ with values in $\mathbb{C}$. The fact that the assumption \eqref{cc1} is weaker than that of Theorem \ref{T1} is justified in \cite{EJ2024,H2023}.

\begin{theo}\label{thm2.2}
Let $F$ be a locally biholomorphic mapping on $\mathbb{B}$, and suppose that $F$ satisfies the assumption \eqref{cc1}. If $F\in \widehat{\mathcal{S}}_{\beta}(\mathbb{B})$, $|\beta|\leq \frac{\pi}{3}$, then for every $x_0\in X$ with $\|x_0\|=1$, we have
\[
|H_{2,2}(F)|\leq \cos^2 \beta,
\]
where $H_{2,2}(F)$ is the second Hankel determinant given by $H_{2,2}(F)=
\begin{vmatrix}
 A_2 & A_3\\
 A_3 & A_4\\
\end{vmatrix},$
or equivalently, $H_{2,2}(F)=A_2A_4-A_3^2,$ where $A_1=1$, and $A_2$, $A_3$, and $A_4$ are defined by
\eqref{eq:An}, respectively. Moreover, the above estimate is sharp.
\end{theo}
\begin{proof} Fix $x_0\in \partial B$ and let $T_{x_0}\in T(x_0)$. Define
\bea\label{aaa1}
\phi(\xi)=
\begin{cases}
\dfrac{\xi}{T_{x_0}\!\left(\varphi(\xi x_0)\right)}, & \xi\neq 0,\\[2mm]
1, & \xi=0,
\end{cases}
\eea
where
\[
\varphi(x)=(DF(x))^{-1}F(x).
\]
Then $\phi\in H(\mathbb U)$ and $\phi(0)=1$. Moreover, since
$F\in \widehat{S}_{\beta}(\mathbb{B})$ , it follows from the definition of spirallikeness that
\[
\Re\!\left(e^{i\beta}\frac{\xi}{T_{x_0}\!\left(\varphi(\xi x_0)\right)}\right)
=
\Re\!\left(e^{i\beta}\frac{\|\xi x_0\|}{T_{\xi x_0}\!\left(\varphi(\xi x_0)\right)}\right)>0,
\qquad \xi\in\mathbb U\setminus\{0\}.
\]
Hence,
\[\Re\!\left(e^{i\beta}\phi(\xi)\right)>0,\qquad \xi\in\mathbb U,
\]
and therefore $e^{i\beta}\phi$ belongs to the Carath\'odory class ($\mathcal{P}$). Then there exists $p\in \mathcal{P}$ such that 
\bea\label{kp1}e^{i\beta}\phi(z)=\cos \beta p(z)+i\sin \beta.\eea
Since $\phi$ is holomorphic function on $\mathbb{U}$, then it has Taylor series expansion:
\bea\label{phi1}\phi(\xi)=1+d_1\xi+d_2\xi^2+d_3\xi^3+d_4\xi^4+\cdots,\xi\in \mathbb{U}.\eea
Using (\ref{aaa1}), \eqref{eq:P}, \eqref{kp1} and \eqref{phi1}, a direct computation yields
\bea\label{qq1}
\begin{cases} \frac{T_{x_0}(D^2\varphi(0)(x_0^2))}{2!}&=-\cos \beta e^{-i\beta}p_1,\\
 \frac{T_{x_0}(D^3\varphi(0)(x_0^3))}{3!}&=\cos\beta\,e^{-i\beta}\left(\cos\beta\,e^{-i\beta}p_1^2-p_2\right)\\
\frac{T_{x_0}(D^4\varphi(0)(x_0^4))}{4!}&=\cos\beta\,e^{-i\beta}\left(-\cos^2\beta\,e^{-2i\beta}p_1^3+2\cos\beta\,e^{-i\beta}p_1p_2-p_3\right).
\end{cases}
\eea
On the other hand, since $\varphi(x)=DF(x))^{-1}F(x)$, it follows that
\bea\label{qqq2}
\begin{cases}
\frac{D^{2}F(0)(x^{2})}{2!}&=-\frac{D^{2}\varphi(0)(x^{2})}{2!}\\
-\frac{D^{3}F(0)(x^{3})}{3}&=\frac{D^{3}\varphi(0)(x^{3})}{3!}+D^{2}F(0)\!\left(x,\,\frac{D^{2}\varphi(0)(x^{2})}{2!}\right)\\
-\frac{D^{4}F(0)(x^{4})}{8}&=\frac{D^{4}\varphi(0)(x^{4})}{4!}+D^{2}F(0)\!\left(x,\,\frac{D^{3}\varphi(0)(x^{3})}{3!}\right)\\
&\;\;\;+\frac{1}{2}D^{3}F(0)\!\left(x^{2},\,\frac{D^{2}\varphi(0)(x^{2})}{2!}\right).
\end{cases}\eea

Combining(\ref{qq1}) and (\ref{qqq2}) together with assumption \eqref{cc1}, we derive explicit formulas for the coefficients \(A_2\), \(A_3\)  and \(A_4\) in terms of the parameters \(p_1\), \(p_2\) and  \(p_3\), respectively.

\bea\label{A2}A_2&=T_{x_0}\!\left(\frac{D^2F(0)(x_0^2)}{2!}\right)=H_{F,1}(x_0)
=-\,T_{x_0}\!\left(\frac{D^2\varphi(0)(x_0^2)}{2!}\right)
=\cos \beta e^{-i\beta}p_1
\eea
and
\[
\frac{D^2\varphi(0)(x_0^2)}{2!}=-\,\frac{D^2F(0)(x_0^2)}{2!}=-H_{F,1}(x_0)x_0
=-A_2x_0
=-\cos \beta e^{-i\beta}p_1x_0.
\]

\bea\label{A3} A_3&=&T_{x_0}\!\left(\frac{D^3F(0)(x_0^3)}{3!}\right)
=H_{F,2}(x_0) \nonumber\\
&=&-\frac{1}{2}\left(T_{x_0}\!\left(\frac{D^3\varphi(0)(x_0^3)}{3!}\right)
+T_{x_0}\!\left(D^2F(0)\!\left(x_0,\frac{D^2\varphi(0)(x_0^2)}{2!}\right)\right)\right) \nonumber\\
&=&-\frac{1}{2}\left(\cos\beta\,e^{-i\beta}\left(\cos\beta\,e^{-i\beta}p_1^2-p_2\right)+T_{x_0}\!\left(D^2F(0)(x_0,-\cos \beta e^{-i\beta}p_1x_0)\right)\right) \nonumber\\
&=&-\frac{1}{2}\left(\cos\beta\,e^{-i\beta}\left(\cos\beta\,e^{-i\beta}p_1^2-p_2\right)-2\cos \beta e^{-i\beta}p_1A_2\right)\nonumber\\
&=&\frac{1}{2}\left(\cos\beta\;e^{-i\beta}\left(\cos\beta\;e^{-i\beta}p_1^2+p_2\right)\right).
\eea

and

\beas 
\frac{D^3\varphi(0)(x_0^3)}{3!}&=&-\frac{D^3F(0)(x_0^3)}{3}-D^2F(0)\!\left(x_0,\frac{D^2\varphi(0)(x_0^2)}{2!}\right)\\
&=&-\frac{D^3F(0)(x_0^3)}{3}-D^2F(0)(x_0,-\cos \beta e^{-i\beta}p_1x_0)\\
&=&\left(-2H_{F,2}(x_0)+2\cos^2\beta e^{-2i\beta}p_1^2\right)x_0\\
&=&\left(-\cos\beta\;e^{-i\beta}\left(\cos\beta\;e^{-i\beta}p_1^2+p_2\right)+2\cos^2\beta e^{-2i\beta}p_1^2\right)x_0\\
&=&\cos\beta\;e^{-i\beta}\left(\cos\beta\;e^{-i\beta}p_1^2-p_2\right)x_0.\eeas
\bea\label{A4}
A_4&=& T_{x_0}\!\left(\frac{D^4F(0)(x_0^4)}{4!}\right)=H_{F,3}(x_0) \nonumber\\
&=&-\frac{1}{3}\Bigg(T_{x_0}\!\left(\frac{D^4\varphi(0)(x_0^4)}{4!}\right)+T_{x_0}\!\left(D^2F(0)\!\left(x_0,\frac{D^3\varphi(0)(x_0^3)}{3!}\right)\right) \nonumber\\
&& +\frac{1}{2}T_{x_0}\!\left(D^3F(0)\!\left(x_0^2,\frac{D^2\varphi(0)(x_0^2)}{2!}\right)\right)\Bigg) \nonumber\\
&=&-\frac{1}{3}\Bigg(T_{x_0}\!\left(\frac{D^4\varphi(0)(x_0^4)}{4!}\right)+\cos\beta\,e^{-i\beta}\left(\cos\beta\,e^{-i\beta}p_1^2-p_2\right)
T_{x_0}\!\left(D^2F(0)(x_0,x_0)\right) \nonumber\\
&&-\frac{1}{2}\cos\beta\,e^{-i\beta}p_1T_{x_0}\!\left(D^3F(0)(x_0^2,x_0)\right)\Bigg) \nonumber\\
&=&-\frac{1}{3}\Bigg(\cos\beta\,e^{-i\beta}\Big(-\cos^2\beta\,e^{-2i\beta}p_1^3+2\cos\beta\,e^{-i\beta}p_1p_2-p_3\Big) \nonumber\\
&&+2A_2\cos\beta\,e^{-i\beta}\left(\cos\beta\,e^{-i\beta}p_1^2-p_2\right)-3A_3\cos\beta\,e^{-i\beta}p_1\Bigg) \nonumber\\
&=&\frac{\cos\beta\;e^{-i\beta}}{6}\left(\cos^2\beta\;e^{-2i\beta}p_1^3+3\cos\beta\;e^{-i\beta}p_1p_2+2p_3\right).
\eea
Combining \eqref{A2}--\eqref{A4} and employing the inequality
$|xa+yb|\le |x|\,|a|+|y|\,|b|$, together with the identity
$|e^{in\beta}|=1$ for every $n\in\mathbb{R}$, we obtain
\beas
|H_{2,2}(f)|
&=&
\left|A_2A_4-A_3^2\right|\\
&\le&
\frac{\cos^2\beta}{12}
\left|
4p_1p_3-3p_2^2-\cos^2\beta\,p_1^4
\right|.
\eeas
The remaining part of the proof proceeds exactly as in the proof of Theorem A (see \cite[Theorem 3.1]{KR2015}); hence, the details are omitted. The sharpness of the result is an immediate consequence of Theorem \ref{T1}. This completes the proof.

\end{proof}

\section*{{\bf Declarations}}
\subsection*{Data Availability Statement}
Data sharing is not applicable to this article as no datasets were generated or analyzed during the current study.
\subsection*{Conflict of Interest}
The authors declare that they have no conflict of interest. 
\subsection*{Author Contributions}
Both authors contributed equally to this work.

\end{document}